\documentclass[12pt]{amsart}
\usepackage{latexsym, amsfonts, amsmath, amscd, esint}
\usepackage{mathrsfs}
\usepackage[bookmarksnumbered,colorlinks,plainpages,backref]{hyperref}
\usepackage[lmargin=2.5cm,tmargin=2cm,bmargin=2cm,rmargin=2.5cm]{geometry}
\usepackage{amssymb}

\usepackage[dvips]{graphicx}
\usepackage{comment}
\usepackage{ulem}

\begin{document}

\newtheorem{definition}{Definition}[section]
\newtheorem{proposition}{Proposition}[section]
\newtheorem{theorem}{Theorem}[section]
\newtheorem{observation}{Observation}[section]
\newtheorem{exercise}{Exercise}
\newtheorem{lemma}{Lemma}[section]
\newtheorem{remark}{Remark}[section]
\newtheorem{corollary}{Corollary}[section]

\newtheorem{result}{Result}
\newtheorem{conjecture}{Conjecture}
\newtheorem{claim}{Claim}
\newtheorem{note}{Note}
\newtheorem{question}{Question}
\newtheorem{fact}{Fact}
\newtheorem{assumption}{Assumption}
\newtheorem{principle}{Principle}
\numberwithin{equation}{section}


\newcommand\norma[1]{\left\lVert#1\right\rVert}

\newcounter{alphatheo}
\newenvironment{alphatheo}{\refstepcounter{alphatheo}\medskip\noindent{\bf Theorem\ \Alph{alphatheo}.} \it }{\medskip}

\newcounter{example}
\newenvironment{example}{\refstepcounter{example}\medskip\noindent{\sc Example\ \thesection.\theexample:}}{\medskip}

\renewenvironment{proof}{\medskip\noindent{\sc Proof:}}{\medskip}

\def\qed{\ifhmode\unskip\nobreak\fi\quad 
  \ifmmode\square\else$\square$\fi}

\newcommand{\R}{\mathbb R}
\newcommand{\Q}{\mathbb Q}
\newcommand{\N}{\mathbb N}
\newcommand{\Z}{\mathbb Z}
\newcommand{\C}{\mathbb C}
\newcommand{\Ss}{\mathbb S}

\def\B{{\mathcal B}}
\def\bsL{{\mathcal L}}
\def\calD{{\mathcal L}}
\def\Cl{{\mathcal C}}
\def\D{{\mathcal D}}
\def\E{{\mathcal E}}
\def\F{{\mathcal F}}
\def\H{{\mathcal H}}
\def\I{{\mathcal I}}
\def\L{{\mathcal L}}
\def\M{{\mathcal M}}
\def\O{{\mathcal O}}
\def\P{{\mathcal P}}
\def\S{{\mathcal S}}
\def\U{{\mathcal U}}
\def\V{{\mathcal V}}
\def\W{{\mathcal W}}
\def\X{{\mathcal X}}
\def\Y{{\mathcal Y}}
\def\ccinf{C^\infty_{c}}
\def\cinf{C^\infty}
\def\bmo{\hbox{\rm bmo\,}}
\def\vmo{\hbox{\rm vmo\,}}
\def\VMO{\hbox{\rm VMO\,}}
\def\<{\langle}
\def\>{\rangle}
\def\eps{\varepsilon}
\def\setadelta{\stackrel{\delta_p}{\longrightarrow}}

\def\ele{\mathcal L}
\def\eleort{\mathcal L^\bot}

\newcommand{\supp}{\text{supp}\,}
\newcommand{\dist}{{\rm dist }}
\newcommand{\diam}{\rm {diam}\,}
\newcommand{\diver}{{\rm div\,}}
\newcommand{\Op}{{\rm Op\,}}
\newcommand\inner[2]{\langle #1, #2 \rangle} 


\title[Nonhom. div-curl type estimates for system of complex vector fields]{Nonhomogeneous div-curl type estimates for system of complex vector fields on local Hardy space}

\author {C. Machado}
\address{Departamento de Matem\'atica, Universidade Federal de S\~ao Carlos, S\~ao Carlos, SP,
13565-905, Brasil}
\email{catarina.machado@estudante.ufscar.br}

\author {T. Picon}
\address{Departamento de Computa\c{c}\~ao e Matem\'atica, Universidade de S\~ao Paulo, Ribeir\~ao Preto, SP,  14040-901 , Brasil}
\email{picon@ffclrp.usp.br}

\thanks{The authors were supported by Funda\c{c}\~ao de Amparo \`a  Pesquisa do Estado de S\~ao Paulo (FAPESP - grant 18/15484-7 and grant 21/12655-8)  
and the second by Conselho Nacional de Desenvolvimento Cient\'ifico e Tecnol\'ogico (CNPq - grant 315478/2021-7)}
\subjclass[2020]{Primary 35J46 35B45 42B30; Secondary 30H10 35A23 46F10}

\keywords{elliptic systems, div-curl estimates; Hardy spaces, bmo decomposition}

\begin{abstract}

In this work, we present a nonhomogeneous version of the classical div-curl type estimates in the setup of elliptic system of complex vector fields with constant coefficients on local Hardy space $h^1$. 
As an application, we obtain a decomposition of the local $bmo$ space via a family of vector fields depending on div-curl terms.

\end{abstract}

\maketitle

\section{Introduction}

The div-curl inequality due to Coifman, Lions, Meyer \& Semmes in \cite{CLMS} asserts that if $V \in L^{p}(\R^{N},\R^{N})$ and $W \in L^{p'}(\R^{N},\R^{N})$ are vector fields satisfying ${\rm div}\;V=0$ and ${\rm curl}\;W=0$, in the sense of distributions, for some $1<p<\infty$ with $\frac{1}{p} + \frac{1}{p'} = 1$, then $V \cdot W$ belongs to the Hardy space $H^1(\R^N)$ and moreover there exists a constant $C>0$ such that
\begin{equation}\label{eq01}
\|V \cdot W\|_{H^{1}}\leq C \|V\|_{L^{p}}\|W\|_{L^{p'}}.
\end{equation}
The previous inequality improves the control obtained by the H\"older inequality, since the Hardy space $H^1(\R^{N})$ is continuously and strictly embedded in $L^1(\R^{N})$. The assumption ${\rm curl}\;W=0$ implies that $W=\nabla \phi$ and the estimate \eqref{eq01} can be written equivalently as 
\begin{equation}\label{eq02}
\| V \cdot \nabla\phi \|_{H^{1}}\leq C \|V \|_{L^{p}}\|\nabla\phi\|_{L^{p'}},
\end{equation} 
where the required condition ${\rm div}\;V=0$ is understood as $V$ belonging to the kernel of the formal adjoint of gradient operator $\nabla$. An extension of this inequality, in the local setting of higher order elliptic linear differential operators with complex variable coefficients, was recently present by the authors in \cite{MP}.

A natural questions arises on a nonhomogeneous version of the inequality \eqref{eq01} when the assumptions on divergence and curl are not free. The answer was presented by Dafni in \cite[Theorem 5]{D5} and we state as following: 

\begin{theorem}\label{prop_1}
	Suppose $V$ and $W$ are vector fields on $\R^N$ satisfying
	$$V \in L^p(\R^{N},\R^{N}), \; W \in L^{p'}(\R^{N},\R^{N}), \; 1 < p < \infty \,\,\, \text{and} \,\,\, \dfrac{1}{p}+\dfrac{1}{p'} =1.$$
	If there exists a function $ f \in L^p(\R^N)$ and a matrix-valued function $A$ with components in $L^{p'}(\R^N)$ such that, in the sense of distributions,
	$${\rm div}\; V = f, \;\;\;\; {\rm curl}\; W = A,$$
	then $V \cdot W$ belongs to the local Hardy space $h^1(\R^N)$, with
\begin{equation}\label{galia}
\norma{V \cdot W}_{h^1} \leq C \left( \norma{V}_{L^p}\norma{W}_{L^{p'}} + \norma{f}_{L^p}\norma{W}_{L^{p'}} + \norma{V}_{L^p} \norma{A}_{L^{p'}} \right)
\end{equation}
\end{theorem}

Here $h^p(\R^{N})$ denotes the local Hardy spaces introduced by Goldberg in \cite{G}. For a given $\varphi \in \mathcal{S}(\R^N)$ such that $\int{\varphi(x)dx} \neq 0$ and for $t>0$, let $\varphi_t(x) := t^{-n}\varphi(t^{-1}x)$. We say that a tempered distribution  $f \in \mathcal{S}'(\R^N)$ belongs to  $h^p(\R^N)$ when
$$
\|f\|_{h^p} := \|m_{\varphi}\|_{L^p} < \infty, \;\; \text{where} \;\;	m_{\varphi}f(x) := \sup_{0<t<1} \left| \langle f,\, \varphi_t(x-\cdot)  \rangle \right|.
$$
The functional $\| \cdot \|_{h^p}$ defines a norm for $p\geq 1$ and a quasi-norm otherwise. We refer to it always as a norm for simplicity. Even though we start with a fixed $\varphi$, the local Hardy spaces remains the same no matter which $\varphi$ we choose. It is well known that $H^{p}(\R^{N})$ is continuously embedded in $h^{p}(\R^{N})$ for all $0<p<\infty$ and $H^{p}(\R^{N})=h^{p}(\R^{N})=L^{p}(\R^{N})$ with comparable norms for $1<p<\infty$. In particular, $H^{1}(\R^{N})\subset h^{1}(\R^{N}) \subset L^{1}(\R^{N})$ strictly. In contrast to Hardy space, the localizable version is closed by test functions, precisely: if $\varphi \in C_{c}^{\infty}(\R^{N})$ and  $f \in h^{p}(\R^{N})$ then $\varphi f \in h^{p}(\R^{N})$.

Estimates of the type \eqref{galia} were extended in several settings, see for instance \cite{BFG, CDY, CGS,HHP, MP, YYZ}. Suppose now $\L := \left\{L_{1},\dots,L_{n}\right\}$ be a system of linearly independent vector fields with complex coefficients defined on $\R^N$ 
and consider the gradient operator associated with $\L$ given by 
$$\nabla_{\L}\,u := (L_{1}u,\dots,L_{n}u), \quad \text{ for } u\in C^{\infty}(\R^N)$$ 
and its adjoint operator 
$${\rm div}_{\L^{*}}\,v := \sum_{j=1}^{n}L^{*}_{j}v_{j}, \quad \text{ for } v\in C^{\infty}(\R^N,\R^{n}),$$ 
for  $L_{j}^{*}:=\overline{ L_j^{t}}$, where $\overline{L_{j}}$ denotes the vector field obtained from $L_j$ by conjugating its coefficients and $L_j^t$ is the formal transpose of $L_j$. Naturally, we may define the curl operator associated with $\L$ given by matrix 
$${\rm curl}_{\L} v := \left( L_i v_j - L_j v_i \right)_{ij},\quad \text{ for } v\in C^{\infty}(\R^N, \C^n).$$ 
Note that when $n=N$ and $L_{j}=\partial_{x_{j}}$ for $j=1,\dots,n$, we get $\nabla_{\mathcal{L}}=\nabla$, ${\rm div}_{\mathcal{L^{\ast}}} = \rm{div}$, and ${ \rm curl_{\L}} = \rm{curl}$. In this paper, we address the following question: for which systems of vector fields $\L$ 
the global estimate 
\begin{equation}\label{eq_main1}
		\norma{V \cdot W}_{h^1} \leq C \left( \norma{V}_{L^p}\norma{W}_{L^{p'}} + \norma{{\rm div}_{\L^*} \; V}_{L^p}\norma{W}_{L^{p'}} + \norma{V}_{L^p} \norma{{\rm curl}_{\L} \; W}_{L^{p'}} \right)
	\end{equation}
holds? Our main result is the following:
 
\begin{alphatheo}\label{teo_principal1}
Let $\left\{L_{1},\dots,L_{n}\right\}$ be an elliptic system of complex vector fields on $\R^{N}$ with constant complex coefficients with $n \geq 2$. If $V \in L^p(\R^N,\C^n)$ and $W \in L^{p'}(\R^N, \C^n)$ with $1<p<\infty$ satisfy 
$$div_{\L^*}\; V \in L^p(\R^N) \text{ and } curl_{\L} \; W \in L^{p'}(\R^N,\C^{n \times n})$$
then $V \cdot W$ belongs to $h^1(\R^N)$. Moreover, there exists a constant $C>0$ such that $\eqref{eq_main1}$ holds.
\end{alphatheo}

The ellipticity of the system $\left\{L_{1},\dots,L_{n}\right\}$ means that, for any {\it real} 1-form $\omega$ satisfying  $\left\langle \omega, L_{j}\right\rangle=0$ for all $j=1,\dots,n$ implies $\omega=0$, that is equivalent to saying that the second order operator
\begin{equation*}
\Delta_{\L}\; :=\;L_1^* L_1+\cdots+L_n^* L_n
\end{equation*}
is elliptic in the classical sense. 

Local estimates of this type were previously studied in the case $W:=\nabla_{\L}\varphi$ and  ${\rm div}_{\mathcal{L^{\ast}}}\;v=0$ in \cite[Theorem A]{HHP}, where $\mathcal{L}$ is an elliptic system of complex vector fields with smooth variable coefficients, namely: for every point $x_{0} \in \Omega$ there exist an open neighborhood $x_0\in U \subset \Omega$ and a constant $C(U)>0$ such that 
\begin{equation}\label{estimateHHP}
	\|\nabla_{\mathcal{L}}\phi\cdot v\|_{h^1} \le C\|\nabla_{\mathcal{L}}\phi\|_{L^p}\|v\|_{L^{p'}}
\end{equation}
holds for any $\phi\in C_{c}^{\infty}(U, \C)$ and $v\in C^{\infty}_{c}(U,\C^{n})$ satisfying ${\rm div}_{\mathcal{L^{\ast}}}\;v=0$. We remark that $curl_{\L}\,W$ is not necessary null. In fact, 
$$curl_{\L} (\nabla_{\L}\phi)=\left([L_{i},L_{j}]\phi\right)_{i,j},$$  
where $[L_{i},L_{j}]:=L_{i}L_{j}-L_{j}L_{i}$ is the commutator of the vector fields. Clearly, if the vector fields $\left\{L_{1},\dots,L_{n}\right\}$ has constant coefficients then $curl_{\L} W=curl_{\L} (\nabla_{\L}\phi)=0$ and then \eqref{estimateHHP} recover \eqref{eq_main1} locally, assuming ${\rm div}_{\mathcal{L^{\ast}}}\;v=0$.

In the same spirit of \cite[Theorem 5]{D5}, the proof of Theorem A  is simplified by reducing it to two specific cases of the inequality \eqref{eq_main1}. 
The  first is a global nonhomogeneous version of the inequality \eqref{estimateHHP}, namely:
\begin{theorem}\label{prop_curl=0const}
	Let $\L = \{L_1,...,L_n\}$ be an elliptic system of complex vector fields on $\R^N$ with complex constant coefficients with $n \geq 2$. If $V \in L^p(\R^N, \C^n)$ and $div_{\L^*} \; V \in L^p(\R^N)$ with $1<p<\infty$, then the inequality 
	\begin{equation*}\label{eqpropA}
		\norma{V \cdot \nabla_{\L} \phi}_{h^1} \leq C \left( \norma{V}_{L^p} + \norma{div_{\L^*} \; V}_{L^p} \right)\norma{\nabla_{\L} \phi}_{L^{p'}}
	\end{equation*}
	holds for all function $\phi$ such that $\nabla_{\L} \phi \in L^{p'}(\R^N, \C^n)$.
\end{theorem}

The second simplification is a reduction of the inequality \eqref{eq_main1} for general $W \in L^{p'}(\R^N, \C^n)$ and  $div_{\L^{\ast}} V=0$.	

\begin{theorem}\label{prop_div=0const}
	Let $\L = \{L_1,...,L_n\}$ be an elliptic system of complex vector fields on $\R^N$ with complex constant coefficients with $n \geq 2$. If $W \in L^{p'}(\R^N, \C^n)$ and $curl_{\L} \;W \in L^{p'}(\R^N)$ with $1<p<\infty$, then the inequality 
	$$\norma{V \cdot W}_{h^1} \leq C  \norma{V}_{L^p}\left( \norma{W}_{L^{p'}} +  \norma{curl_{\L}W }_{L^{p'}} \right)$$
	holds for all $V \in L^p(\R^N, \C^n)$ which satisfies $div_{\L^{\ast}} V=0$.	
\end{theorem}

The conclusion of \eqref{eq_main1} will follow  by a Hodge type decomposition $V=V_{1}+V_{2}$ given by Lemma \ref{hodge} for each $V \in L^{p}(\R^{N},\C^{n})$, in which $div_{\L^{\ast}} V_1 =0$ and $V_{2}=\nabla_{\L}\phi$.



In \cite{CLMS}, the authors proved a type of converse of inequality \eqref{eq01}, called div-curl lemma, that asserts each $f \in H^{1}{(\R^{N})}$ can be written as
$$f=\sum_{j=1}^{\infty}\lambda_{k}f_{k}$$ 
in the sense of distribution, where the sequence $\left\{ \lambda_{k} \right\}_{k} \in \ell^{1}(\R)$ and $f_{k}:=V_{k}\cdot W_{k}$ with $W_{k},V_{k} \in L^{2}(\R^{N},\R^{N})$ satisfying $\rm div \, V_k =0$ and
$\rm curl \, W_k =0$. This result is a direct consequence from the duality $BMO(\R^{N})=(H^{1}(\R^{N}))^{\ast}$ and a characterization of the $BMO$ norm given by
\begin{equation}
\|g\|_{BMO} \approx \sup_{V,W} \int_{\R^{N}} g(x) (V \cdot W)(x)dx,
\end{equation}
where the supremum is taken all vector fields  $W,V \in L^{2}(\R^{N},\R^{N})$ satisfying $\rm div \, V =0$, $\rm curl \, W =0$ and $\|V\|_{L^{2}}, \|W\|_{L^{2}} \leq 1$. So now, let $\L = \left\{L_{1},\dots,L_{n}\right\}$ as in the statement of Theorem A and denote by $(\D C_{\L})^p_{{0,1}}$ the family of all functions which can be written in the form $V \cdot W$, where $V \in L^p (\R^N, \C^n)$ and $W \in L^{p'}(\R^N, \C^n)$ are vector fields satisfying $\norma{V}_{L^p}, \norma{W}_{L^{p^{\prime}}}  \leq 1$
with ${\rm div}_{\L^*} \; V =0$ and $\norma{{\rm curl}_{\L} \; W}_{L^{p'}} \leq 1$. Analogously, we define 
$(\D C_{\L})^p_{{1,0}}$ the family of all functions $V \cdot W$, where $V \in L^p (\R^N, \C^n)$ and $W \in L^{p'}(\R^N, \C^n)$ are vector fields satisfying $\norma{V}_{L^p}, \norma{W}_{L^{p^{\prime}}}  \leq 1$
with $\norma{{\rm div}_{\L^*} \; V}_{L^{p1}} \leq 1$ and $W:=\nabla_{\L}\phi$.

Our second main result is the following:

\begin{alphatheo}\label{teo_principal2}
Let $\L = \{L_1,...,L_n\}$ be an elliptic system of complex vector fields on $\R^N$ with complex constant coefficients with $n \geq 2$. 
If $g \in bmo(\R^N)$, then 
	$$\norma{g}_{bmo} \; \simeq \sup_{f \in (\D C_{\L})^p_{1,0}} \; \left| \int_{\R^N} g(x)f(x)dx \right| \; \simeq \sup_{f \in (\D C_{\L})^p_{0,1}} \; \left|\int_{\R^N} g(x)f(x)dx\right|,$$
for any $1<p<\infty$.
\end{alphatheo}

We recall the dual of $h^{1}(\R^{N})$ can be identified with the space $bmo{(\R^{N})}$ given by the set of locally integrable functions $f$ that satisfy
\begin{equation}\label{bmo}
\|g\|_{bmo}:=\sup_{|B|\leq 1}\fint_{B}|g(x)-g_{B}|dx+ \sup_{|B|> 1}\fint_{B}|g(x)|dx<\infty,
\end{equation}
 where $\displaystyle{g_{B}:=\frac{1}{|B|}\int_{B}g(x)dx}$. 

As a direct consequence of the previous characterization and duality, we announce the following div-curl lemma associate to an elliptic system of complex vector fields. 

\begin{corollary}\label{coro1}
Let $\L = \{L_1,...,L_n\}$ be an elliptic system of complex vector fields on $\R^N$ with complex constant coefficients with $n \geq 2$ and $1<p<\infty$. For each $f \in h^1(\R^N)$ there exist a sequence $\{\lambda_k\}_k \in \ell^1(\C)$ 
and a sequence  $\left\{f_k \right\}_{k} \in (\D C_{\L})^p_{1,0}$ such that
\begin{equation}\label{coroC}
f = \sum_{k=1}^{\infty} \lambda_k f_k,
\end{equation}
in the sense of distributions. The same decomposition holds replacing $(\D C_{\L})^p_{1,0}$ by $(\D C_{\L})^p_{0,1}$.
\end{corollary}

The organization of the paper is as follows. In Section \ref{S1}, we recall some definitions, elliptic estimates and a Hodge  decomposition associated with system of complex vector fields. The Section \ref{S4} is devoted to prove of Theorem A as consequence of the Theorems \ref{prop_curl=0const} and \ref{prop_div=0const}. In the Section \ref{S6}, we present the proof of Theorem B and, in the end of the section, the proof of Corollary \ref{coro1}. 

\bigskip 
\noindent \textbf{Notations}.
Throughout the paper we will use the notation $\Omega \subset \R^{N}$ for an open set and by $B_x^t$ for an open ball $B(x,t)$ centered at $x$ and radius $t>0$ ($B$ denotes a generic ball). {We use the multi-index derivative notation $\partial^\alpha$ to denote $\displaystyle{\frac{\partial^{|\alpha|}}{\partial_{x_1}^{\alpha_1} \partial_{x_2}^{\alpha_2} \dots \partial_{x_N}^{\alpha_N}}}$, where $\alpha=(\alpha_1, \alpha_2, \dots, \alpha_N) \in \Z_+$ and $|\alpha| := \alpha_1 + \alpha_2 + \dots + \alpha_N$. Furthermore, we also use the simplified notation $\partial^k = \left(\partial^\alpha\right)_{|\alpha|=k}$.} We set $S(\R^{N})$ the Schwartz space and $S'(\R^{N})$ the set of tempered distributions. We denote by $W^{k,p}(\Omega)$ the space of distributions in which all (weak) derivatives with order less or equal than $k$ belongs $L^{p}(\Omega)$ and by $W^{-k,p'}(\Omega)$  its dual space. Here, $p'$ denotes the conjugate exponent to $p$ for $1<p<\infty$ given by $\frac{1}{p}+\frac{1}{p'}=1$. The closure of $C^\infty_c(\Omega)$ in $W^{k,p}(\Omega)$ is denoted by $W^{k,p}_0(\Omega)$.
Another basic notation is the Hardy-Little\-wood maximal operator defined for functions $f \in L^1_{{\rm loc}}(\R^N)$ given by
$$
Mf(x) := \; \sup_{x\in B} \fint_B |f(y)|dy, \quad {\rm a.e.} \ x\in\R^N,
$$
where the supremum is taken over all balls that contain $x$ and  $\fint_B:=\frac{1}{|B|}\int_{B}$, with $|B|$ the Lebesgue measure of $B$.  It is well known that $M:f \mapsto Mf$ is a bounded operator in $L^p(\R^N)$ for $1<p\le \infty$, and for $f\in L^\infty(\R^N)$ we have the trivial estimate $Mf(x) \le \|f\|_{\infty}$, almost everywhere  $x\in\R^N$.



\section{Elliptic system of complex vector fields}\label{S1}

Consider n complex vector fields $\L:=\left\{L_1, ... , L_n\right\}$, $n \geq 2$, with constant complex coefficients in $\R^N$ for $N \geq 2$. We will always assume that
\begin{enumerate}
\item[(a)] $\left\{L_1, .... , L_n\right\}$ are everywhere linearly independent;
\item[(b)] the system $\left\{L_1, ... , L_n\right\}$ is elliptic.
\end{enumerate}
We recall this means that, for any {\it real} 1-form $\omega$ satisfying  $\left\langle \omega, L_{j}\right\rangle=0$ for all $j=1,\dots,n$ implies $\omega=0$.
Consequently, the number n of vector fields must satisfy $N/2 \leq n \leq N$.
{Alternatively, (b) is equivalent to saying that the real homogeneous differential operator with order two 
\begin{equation*}
\Delta_{\L}\; :=\;L_1^* L_1+\cdots+L_n^* L_n
\end{equation*}
is elliptic in the classical sense, where $L^{*}_{j}=-\overline{L_{j}}$ is the formal adjoint of $L_{j}$.} 
We remark that choosing an appropriate generators and reordering the coordinates $\left\{x_1, x_2,...,x_N\right\}$, we always may assume without loss of generality that the vector fields $\{L_1,...,L_n\}$ have the form
 \begin{equation}\label{Lsimples}
L_{j}=\frac{\partial}{\partial x_{j}}+\sum_{k=1}^{m}a_{jk}\frac{\partial}{\partial x_{{n+k} }}, 
 \end{equation}
for $j=1,...,n$ with $m:=N-{n}$. {The ellipticity of $\Delta_{\mathcal{L}}$ means that there exists $C>0$ such that 
\begin{align*}
\sum_{j=1}^{n}\left| \xi_{j}+\sum_{k=1}^{m}a_{jk}\xi_{n+k}\right|^{2} \geq C |\xi|^{2}, \quad \forall \, \xi \in \R^{N}.
\end{align*}}
Note that $\Delta_{\L}$ is a slight variation of Laplacian operator and it has a fundamental solution $E(x)$ i.e.  $\Delta_{\L} E =\delta_{0}$
 that is locally integrable tempered distribution homogeneous of degree $-N+2$ for $N \geq 3$ and $\log |x|$ type for $N=2$. In particular, $\partial^{2}E$ is a bounded operator from $L^{p}(\R^{N})$ to itself for $1<p<\infty$.   

 An important class of elliptic system satisfying \eqref{Lsimples}  is given by
\begin{align*}
L_{j}=\frac{\partial}{\partial x_{j}}+i\frac{\partial}{\partial x_{r+j}}, \,\,\text{for}\,\, j=1,...,r \quad \text{and} \quad L_{2r+j}=\frac{\partial}{\partial x_{2r+j}}, \,\,\text{for}\,\, j=1,...,s
\end{align*}
where $N=2r+s$. When $s=0$ we obtain the Cauchy-Riemman system in  $\C^{r}\cong \R^{2r}$. Note that, in this particular case, $\Delta_{\L}$ is a multiple of Laplacian operator $\Delta$ (see \cite{BCH}).

\begin{lemma}\label{calderon}
Let $\L=\{L_1,L_2,..., L_n\}$ be an elliptic system of complex vector fields on $\R^{N}$ with constant complex coefficients and $1<p<\infty$. Then there exists $C>0$ such that 
\begin{equation}\label{eliptic2}
\|\nabla \phi\|_{L^{p}}\leq C \|\nabla_{\L }\phi\|_{L^{p}}, \quad \forall \,\,\nabla \phi \in L^{p}(\R^{N}).
\end{equation}
\end{lemma} 

\begin{proof}
 Using the fundamental solution of $\Delta_{\L}$ and that the vector fields have constant coefficients, we may write $\nabla \phi=\nabla div_{\L^{\ast}}{\left( E \ast \nabla_{\L}\phi\right) = (\nabla div_{\L^*} \; E) \ast \nabla_{\L}\phi}$ and since $\partial^{2}E$ is a bounded operators on $L^{p}(\R^{N})$ for $1<p<\infty$ the estimate \eqref{eliptic2} follows. \qed
\end{proof}

Next we present a Hodge decomposition for vector fields in our div-curl setting:

\begin{lemma}\label{hodge}
Let $\L=\{L_1,L_2,..., L_n\}$ be an elliptic system of complex vector fields on $\R^{N}$ with constant complex coefficients. Each $V \in L^{p}(\R^{N},\C^{n})$ for $1<p<\infty$ can be decomposed as
$$V=V_{1}+V_{2},$$ 
with $div_{\L^{\ast}} V_1 =0$, $V_{2}=\nabla_{\L}\varphi_{2}$. Moreover
\begin{equation}\label{eliptic}
\|V_{i}\|_{L^{p}}\lesssim \|V\|_{L^{p}}, \quad \text{for}\,\,\, i=1,2.
\end{equation}
\end{lemma} 

\begin{proof}
The proof is standard. Using the fundamental solution of $\Delta_{\L}$, we may define
$V_{2}:= \nabla_{\L} \varphi_{2}$ with $\varphi_{2}= E \ast div_{\L^*}V$ and $V_{1}:=V-V_{2}$. Clearly 
$$div_{\L^*}V_{2}=div_{\L^*}\nabla_{\L} \varphi_{2}=\Delta_{\L}\varphi_{2}=div_{\L^*}V,$$
thus $div_{\L^*}V_{1}=0$. The estimate \eqref{eliptic} follows directly from $\partial^{2}E$ is a bounded operator from $L^{p}(\R^N)$ to itself for $1<p<\infty$. \qed
\end{proof}

Now we state an important \textit{a priori} estimate that will be useful in this work.
 
 \begin{lemma}\label{lemma_normaL_W}
Consider  $\L = \{L_1,...,L_n\}$ be a system of complex vector fields on $\R^N$ with complex constant coefficients {and $1<r<\infty$}. 
Then for each ball $B \subset \R^N$, there is a constant $C=C(B, \L)>0$ such that 
\begin{equation}\label{eqforte}
\norma{\nabla g}_{W^{-1,r}(B)} \leq C \sum_{j=1}^n \norma{L_j^* g}_{W^{-1,r}(B)}.
 \end{equation}
 \end{lemma}
 
 We recall that $\displaystyle{\norma{\nabla g}_{W^{-1,r}(B)}:=\sum_{j=1}^{N}\norma{\partial_{x_{j}} g}_{W^{-1,r}(B)}}$, where
 $$\norma{ \partial_{x_i} g}_{W^{-1,r}(B)} = \sup_{\substack{\norma{ u}_{W^{1,r'}(B)} \leq 1\\u \in C^\infty_c(B)}} |\inner{ g}{\partial_{x_i}u}|= \sup_{\substack{\norma{ u}_{W^{1,r'}(B)} \leq 1\\u \in C^\infty_c(B)}} \left| \int g(x) \; \overline{\partial_{x_i} u(x)}dx \right|.$$

 \begin{proof}
 	Using the fundamental solution of $\Delta_{\L}$ and since the vector fields has constant coefficients, we may write   
 	$\partial_{x_i} u = \sum_{j=1}^{n} {L_j}h_{ij}$ with $h_{ij}:=\partial_{x_i} {L_j^*}E \ast u $. Thus,
 
 	\begin{align*}
 		\left| \int_B g(x) \; \overline{\partial_{x_i} u(x)}dx \right| 
		\leq \sum_{j=1}^{n} \left| \int_B g \;  \overline{L_jh_{ij}(x)}dx \right|
 		&= \sum_{j=1}^{n} \left| \langle{L_j^*}g, \chi_{B}h_{ij} \rangle  \right| \\
 		&\leq \sum_{j=1}^{n} \norma{{L_j^*} g}_{W^{-1,r}(B)} \norma{h_{ij}}_{W^{1,r'}(B)}.
 	\end{align*}
As $\displaystyle{h_{ij}=\sum_{k=1}^{N}-\overline{a_{kj}} \left(\partial^{2}_{x_{i}x_{k}}E \ast u \right)}$ and noting that $\|\partial^{2}_{x_{i}x_{k}}E \ast u\|_{L^{r'}}\leq C_{ik}\|u\|_{L^{r'}}$,  	
we have
 	\begin{eqnarray*}
 		\left| \int_B g(x) \; \overline{\partial_{x_i} u(x)}dx \right| 
 		\lesssim \sum_{j=1}^{n} \norma{{L_j^*} g}_{W^{-1,r}(B)} \norma{u}_{W^{1,{r'}}(B)}
 	\end{eqnarray*}
for all $u \in C^\infty_c(B)$ that implies \eqref{eqforte}. \qed
 \end{proof}



\section{Proof of Theorem A}\label{S4}

In order to obtain the proof of Theorem A, we assume the validity of Theorems \ref{prop_curl=0const} and \ref{prop_div=0const}.  Using the Hodge decomposition from Lemma \ref{hodge}, we may write $V=V_{1}+V_{2}$ and  $W=W_{1}+W_{2}$ with 
$$div_{\L^*} V_1 = div_{\L^*} W_2 = 0 \text{ and } V_2 = \nabla_{\L} \phi_2,  W_1 = \nabla_{\L} \phi_1,$$ 
in the sense of distributions, for some {$\phi_1 \in L^{p'}(\R^N)$ and $\phi_2 \in L^{p}(\R^N)$}. 
Then,
	$$V \cdot W = V_1 \cdot W + V_2 \cdot W_1 + V_2 \cdot W_2,$$
and from Theorem \ref{prop_div=0const} we have

	\begin{eqnarray*}
		\norma{V_{1} \cdot W}_{h^1} 
		&\lesssim& \norma{V_1}_{L^p}\left( \norma{W}_{L^{p'}} + \norma{curl_{\L} W}_{L^{p'}}\right) \\
		&\lesssim& \norma{V}_{L^p}\left( \norma{W}_{L^{p'}} + \norma{curl_{\L} W}_{L^{p'}}\right) 
			\end{eqnarray*}
since $div_{\L^*} V_1 = 0$ and from Theorem \ref{prop_curl=0const} we have
	\begin{eqnarray*}
		\norma{V_{2} \cdot W_{1}}_{h^1} 
		&\lesssim& \norma{W_{1}}_{L^{p'}}\left( \norma{V_{2}}_{L^{p}} + \norma{div_{\L^{}} V_{2}}_{L^{p}}\right) \\
		&=& \norma{W_{1}}_{L^{p'}}\left( \norma{V_{2}}_{L^{p}} + \norma{div_{\L^{}} V}_{L^{p}}\right) \\
		&\lesssim& \norma{W}_{L^{p'}}\left( \norma{V}_{L^{p}} + \norma{div_{\L^{}} V}_{L^{p}}\right) 		
			\end{eqnarray*}
since $W_{1}=\nabla_{\L}\phi_{1}$ and
	\begin{eqnarray*}
		\norma{V_{2} \cdot W_{2}}_{h^1} 
		&\lesssim& \norma{V_{2}}_{L^{p}}\left( \norma{W_{2}}_{L^{p'}} + \norma{div_{\L^{*}} W_{2}}_{L^{p}}\right) \\
		&\lesssim& \norma{V_{2}}_{L^{p}}\norma{W_{2}}_{L^{p'}}  \\
		&\lesssim& \norma{V}_{L^{p}}\norma{W}_{L^{p'}} 
	\end{eqnarray*}
since $V_{2}=\nabla_{\L}\phi_{2}$ and $div_{\L^{*}} W_{2}=0$. Combining the previous estimates, we obtain the desired estimate.
\begin{flushright}
	\qed
\end{flushright}

Fixed  $\varphi \in C_c^{\infty}(B(0,1))$ with $\varphi \geq 0$ and $\int \varphi =1$, denote for each $x \in \R^N$ and $ t>0$ the function $\varphi^x_t(y) := \dfrac{1}{t^N} \varphi \left(\frac{x-y}{t}\right)$. Given $1<s\leq \infty$ and $f \in W^{-1,s}_{loc}(\R^N)$, we define by 
$M^{loc}_{W^{-1,s}}f(x)$ a local maximal operator as the smaller constant $C>0$ which satisfies
\begin{equation}
|\inner{f}{\varphi^x_t(\phi-\phi_{B^x_t})}| \leq C \left( \fint_{B(x,t)} | \bigtriangledown\phi|^{s'} \right)^{\frac{1}{s'}},
\end{equation}
for all $0<t<1$ and  $\phi \in W^{1,s'}_{loc}(\R^N)$. The boundedness of $M^{loc}_{W^{-1,s}}$ on $L^{p}(\R^{N})$ was proved by Dafni in \cite{D5}, precisely:

\begin{lemma}\label{lemma_bound_M}
	If $1<s<p^*$ for $1<p<N$ or $1<s<\infty$ for $p \geq N$ then there exists $C=C(p,s,N)>0$ such that $\norma{M^{loc}_{W^{-1,s}}f}_{L^p} \leq C \norma{f}_{L^p}$, for all $f \in L^{p}(\R^{N})$.
\end{lemma}	

We recall that for each $u \in W^{1,p}(\R^{N})$ with $1\leq p<N$ there exists a constant $C=C(N,p)>0$ such that
	\begin{equation}
		\left( \fint_{B} \left|\dfrac{1}{r_{B}}\left(u-u_{B}\right)\right|^{p^{\ast}} \; \right)^{\frac{1}{p^{\ast}}} \leq C \left( \fint_{B} |\nabla u|^p \right)^{\frac{1}{p}}
	\end{equation}
for any ball $B$ where $r_{B}$ is its radius. This inequality is known as Sobolev-Poincar\'e inequality (see  \cite[Theorem 3, pp. 265]{E}). 

\subsection{Proof of Theorem \ref{prop_curl=0const}}
	Let $\phi$ a function such that $\nabla_{\L} \phi \in L^{p'}(\R^N)$ that is equivalent to $\nabla \phi \in L^{p'}(\R^N)$ from Lemma \ref{calderon}. For each $x \in \R^N$ and $ 0<t<1$ we define
	$$\Phi_t^x(y) := \varphi_t^x(y)(\phi(y)-\phi_{B_x^t})$$
that is supported on $B(x,t)$. From the definition of $div_{\L^*} \; V$, we have $$\inner{div_{\L^*} \; V}{ \Phi^x_t} 
			= \sum_{j=1}^{n}  \langle L_j^* V_j \; \overline{\Phi^x_t} \rangle
			= \int_{B(x,t)} V \cdot \overline{\nabla_{\L} \Phi^x_t}
$$ and taking the product
\begin{align*}
\nabla_{\L} \Phi^x_t(y)&= \nabla_{\L} \left[\varphi_t^x(y)(\phi(y)-\phi_{B_x^t}) \right] \\
&=  \left[ -\dfrac{1}{t^{N+1}} \nabla_{\L} \varphi\left(\dfrac{x-y}{t}\right).(\phi(y)-\phi_{B_x^t}) + \varphi_t^x(y) \; \nabla_{\L} \phi(y) \right]
\end{align*}

that implies
	
		\begin{eqnarray}\label{eq_conv}
			\varphi_t*( V \cdot \nabla_{\L} \phi)(x) &=& \overline{\inner{div_{\L^*} \; V}{ \Phi^x_t}}  \\
			&+& \nonumber \dfrac{1}{t^{N+1}} \int_{B(x,t)} \left[ \nabla_{\L} \varphi\left(\dfrac{x-y}{t}\right) \right] \cdot \left[ \overline{V(y)} ({\phi}(y)-{\phi_{B_x^t}}) \right] dy.
		\end{eqnarray}
		Let $1<\alpha<p$, $1<\beta< p'$ satisfying $\dfrac{1}{\alpha} + \dfrac{1}{\beta} = 1 +\dfrac{1}{N}$. Note that $\beta^*=\alpha'$ and $\beta<N$. 
We point out that $\phi \in L^{\alpha'}_{loc}(\R^{N})$. In fact, if $1<p'<N$ and $\nabla \phi \in L^{p'}(\R^{N})$ then by Sobolev-Gagliardo-Nirenberg inequality $\phi \in L^{p'_{\ast}}(\R^{N})$ with
$$\frac{1}{p'_{\ast}}:=\frac{1}{p'}-\frac{1}{N}<\frac{1}{\beta}-\frac{1}{N}=\frac{1}{\alpha'}$$
that implies $\alpha'<p^{'}_{\ast}$ and consequently $\phi \in L^{\alpha'}_{loc}(\R^{N})$. Otherwise, if $p' \geq N$ then $\phi \in L^{q}_{loc}(\R^{N})$ for any $1 \leq q <\infty$.  Applying the H\"older's inequality and the Sobolev-Poincar\'e inequality,  the second term in \eqref{eq_conv} can be controlled by
\begin{align*}
 \dfrac{\norma{\nabla_{\L} \varphi}_{L^{\infty}}}{t^{N+1}} \int_{B(x,t)} \left| \overline{V(y)} ({\phi}(y)-{\phi_{B_x^t}}) \right| dy
			&\lesssim \left( \fint_{B(x,t)} \left| V(y)\right|^{\alpha} dy \right)^{\frac{1}{\alpha}} \; \left( \fint_{B(x,t)} \left|\dfrac{1}{t}(\phi(y)-\phi_{B_x^t}) \right|^{\alpha'} dy \right)^{\frac{1}{\alpha'}} \\
			&=\left( \fint_{B(x,t)} \left| V(y)\right|^{\alpha} dy \right)^{\frac{1}{\alpha}} \; \left( \fint_{B(x,t)} \left|\dfrac{1}{t}(\phi(y)-\phi_{B_x^t}) \right|^{\beta^*} dy \right)^{\frac{1}{\beta^*}} \\
			&\lesssim \left( \fint_{B(x,t)} \left| V(y)\right|^{\alpha} dy \right)^{\frac{1}{\alpha}} {\left( \fint_{B(x,t)} \left|\nabla \phi(y) \right|^{\beta} dy \right)^{\frac{1}{\beta}}}\\
			& \lesssim  \left[ M(\left| V\right|^{\alpha} )(x) \right]^{\frac{1}{\alpha}}  \left[ M(\left|\nabla \phi \right|^{\beta})(x) \right]^{\frac{1}{\beta}},
\end{align*}	
where $M$ denotes the Hardy-Littlewood maximal function. 
From the definition of $M^{loc}_{W^{-1,s}}(div_{\L^*} \; V)$, the first term \eqref{eq_conv}  is controlled by  
\begin{align*}
\left| \inner{div_{\L^*} \; V}{ \Phi^x_t} \right| &\leq M^{loc}_{W^{-1,s}}(div_{\L^*} \; V)(x) \left( \fint_{B(x,t)} | \nabla \phi(y)|^{s'} dy \right)^{\frac{1}{s'}} \\
&\lesssim  M^{loc}_{W^{-1,s}}(div_{\L^*} \; V)(x) \left[ M \left(| \nabla \phi|^{s'}\right)(x) \right]^{\frac{1}{s'}}, 
\end{align*} 		
for some $1<s<\infty$ to be chosen later. Taking the supremum for $0<t<1$ we have
\begin{equation}\label{fund1}
	m_{\varphi}(V \cdot \nabla_{\L} \phi )(x) \lesssim  M^{loc}_{W^{-1,s}}(div_{\L^*} \; V)(x) \left[ M \left(| \nabla \phi|^{s'}\right)(x) \right]^{\frac{1}{s'}} 
			+ \left[ M(\left| V\right|^{\alpha} )(x) \right]^{\frac{1}{\alpha}}  \left[ M(\left|\nabla \phi \right|^{\beta})(x) \right]^{\frac{1}{\beta}},
\end{equation}
and to compute the norm $\|V \cdot \nabla_{\L} \phi \|_{h^{1}}$ it is sufficient estimate each term in the right side hand in $L^1$ norm. Using the H\"older's inequality for the first term, we have	
\begin{align*}
\| M^{loc}_{W^{-1,s}}(div_{\L^*} \; V)\left[ M \left(| \nabla \phi|^{s'}\right)\right]^{\frac{1}{s'}} \|_{L^{1}} 
&\leq \| M^{loc}_{W^{-1,s}}(div_{\L^*} \; V)\|_{L^{p}} \| \left[M \left(| \nabla \phi|^{s'}\right)\right]^{\frac{1}{s'}}\|_{L^{p'}} \\
& = \| M^{loc}_{W^{-1,s}}(div_{\L^*} \; V)\|_{L^{p}} \| M \left(| \nabla \phi|^{s'}\right)\|^{1/s'}_{L^{p'/s'}} \\
& \leq \| M^{loc}_{W^{-1,s}}(div_{\L^*} \; V)\|_{L^{p}} \| \nabla \phi\|_{L^{p'}}, 
\end{align*}
where in the last inequality we used the boundedness of Hardy-Littlewood maximal function since $p'>s'$ that is equivalent to $p<s$. Note that if $1<p<N$ then we may choose some $p<s<p^{*}$, otherwise if $p \geq N$ we choose any $1<s<p$. Thus from Lemma \ref{lemma_bound_M} we have
\begin{align*}
\| M^{loc}_{W^{-1,s}}(div_{\L^*} \; V)\left[ M \left(| \nabla \phi|^{s'}\right)\right]^{\frac{1}{s'}} \|_{L^{1}} 
& \lesssim \| M^{loc}_{W^{-1,s}}(div_{\L^*} \; V)\|_{L^{p}} \| \nabla \phi\|_{L^{p'}}\\
& \lesssim \| div_{\L^*} \; V\|_{L^{p}} \| {\nabla} \phi\|_{L^{p'}}.
\end{align*}
For the second term, we use the H\"older's inequality and the boundedness of maximal operator M again to conclude that  
\begin{align*}
 \|\left[ M(\left| V\right|^{\alpha}) \right]^{\frac{1}{\alpha}}  \left[ M(\left|\nabla \phi \right|^{\beta}) \right]^{\frac{1}{\beta}}\|_{L^{1}} &\lesssim  \|  M(\left| V\right|^{\alpha})\|^{1/\alpha}_{L^{p/\alpha}}  \| M(\left|\nabla \phi \right|^{\beta})\|^{1/\beta}_{L^{p'/\beta}} 
 \lesssim  \|  V\|_{L^{p}}  \| \nabla \phi \|_{L^{p'}}.
\end{align*}
Combining the previous control in norm $L^{1}$ and using the Lemma \ref{calderon} we have
\begin{align*}
\|V \cdot \nabla_{\L} \phi \|_{h^{1}}& \lesssim   \| div_{\L^*} \; V\|_{L^{p}} \| \nabla \phi\|_{L^{p'}}+ \|  V\|_{L^{p}}  \| \nabla \phi \|_{L^{p'}} 
\lesssim \left( \|  V\|_{L^{p}} +  \| div_{\L^*} \; V\|_{L^{p}} \| \right) \| \nabla_{\L} \phi \|_{L^{p'}}, 
\end{align*} 
as desired.
\qed


\subsection{Proof of Theorem \ref{prop_div=0const}}

Let $V:=(V_1,V_2,...,V_{n})$ and  $U_i:= -E \ast V_{i}$, where $E$ is the fundamental solution of $\Delta_{\L}$. Clearly $-\Delta_{\L} \; U_i = V_i$  and  $\norma{\partial^2 U_i}_{L^{p}} \leq C \norma{V_i}_{L^p}$, for any $1<p<\infty$ and for each $i=1, ... ,n$. Note that $U:=(U_1,U_2,...,U_n)$ satisfies $\; div_{\L^*} \; U = \; div_{\L^*} \; V = 0$. Consider now  $B:= curl_{{\L}} \; U=(B_{ij})_{1\leq i,j\leq n}$ with $B_{ij}:= {L_j}U_i - {L_i}U_j$ and denote $B_j := (B_{1j} \; B_{2j} \; ... \; B_{nj})$ the $j$-th  column of the (symmetric) matrix $B$. Thus 
	\begin{equation*}
		div_{\L^*}\; B_j = \sum_{i=1}^{n} {L_i^*} B_{ij}  
	= {L_j} (div_{\L^*}\; U) - \Delta_{\L} \; U_j = V_j.
	\end{equation*}
	
	This way, 
	\begin{eqnarray*}
		\overline{V \cdot W} = \sum_{j=1}^{n} (\overline{div_{\L^*}\; B_j})W_j &=& -\sum_{j=1}^{n} (div_{\L}\; \overline{B_j})W_j\\
		&=& -\sum_{j=1}^{n} div_{\L}\; (\overline{B_j}W_j) + \sum_{i,j=1}^{n} \overline{B_{ij}} {L_i}(W_j)\\
		&=& -\sum_{j=1}^{n} div_{\L}\; (\overline{B_j}W_j) + \sum_{i<j} \overline{B_{ij}}\left( {L_i}W_j - {L_j}W_i \right) \\
		&=& -\sum_{j=1}^{n} div_{\L}\; (\overline{B_j}W_j) + \sum_{i<j} \overline{B_{ij}} ({\rm curl}_{\L} \; W)_{ij}.
	\end{eqnarray*}
	Now, let $\widetilde{B}$ the symmetric matrix given by $\widetilde{B}_{ij} := B_{ij} - (B_{ij})_{B_x^t}$ that satisfies  $div_{\L^*}\; \tilde{B}_j = div_{\L^*}\; B_j = V_j$. 
It is clear that 
	$$\overline{V \cdot W} = - \sum_{j=1}^{n} div_{\L}\; (\overline{\widetilde{B}_j} W_j) + \sum_{i<j} \overline{\widetilde{B}_{ij}} ({\rm curl}_{\L} \; W)_{ij}.$$

Let $\varphi \in C_c^{\infty}(B(0,1))$ with $\varphi \geq 0$ and $\int \varphi =1$, then we may write 
	\begin{eqnarray*}
		\sum_{i<j} \inner{ ({\rm curl}_{\L} \; W)_{ij}}{\varphi_t^x\widetilde{B}_{ij}} &=& \sum_{i<j} \int_{B(x,t)} ({\rm curl}_{\L} \; W)_{ij}(y)\; \overline{\varphi_t^x(y)\widetilde{B}_{ij}(y)} \; dy\\
		&=& \int_{B(x,t)} \varphi_t^x(y) \sum_{i<j} \overline{\widetilde{B}_{ij}}(y) ({\rm curl}_{\L} \; W)_{ij}(y) \; dy \\
		&=& \int_{B(x,t)} \varphi_t^x(y) \left( \sum_{j=1}^{n} div_{\L}\; (\overline{\widetilde{B}_j}W_j)(y) + \overline{V \cdot W}(y) \right)\; dy\\
		&=& \sum_{j=1}^{n} \int_{B(x,t)} \varphi_t^x(y) \; div_{\L}\; (\overline{\widetilde{B}_j}W_j)(y) \; dy + \overline{\varphi_t * V \cdot W}(x)\\
		&=& -\sum_{j=1}^{n} \int_{B(x,t)} \nabla_{\L} \varphi_t^x(y) \cdot (\widetilde{B}_j \overline{W_j})(y) \; dy + \overline{\varphi_t * V \cdot W}(x). \\
	\end{eqnarray*}
that implies 
\begin{eqnarray*}
		|\left(\varphi_t * V \cdot W\right)(x)| &\leq& \sum_{i<j} \left| \inner{ ({\rm curl}_{\L} \; W)_{ij}}{\varphi_t^x\widetilde{B}_{ij}} \right|
		+ \dfrac{\norma{ \nabla_{\L} \varphi}_{L^{\infty}}}{t^{N+1}} \int_{B(x,t)} \left| (\tilde{B}_j \overline{W_j})(y) \right| dy \\
		&\lesssim& \sum_{i<j} M^{loc}_{W^{-1,s}}(({\rm curl}_{\L} \; W)_{ij})(x) \left(\fint_{B(x,t)} | \nabla \widetilde{B}_{ij}(y)|^{s'} dy \right)^{\frac{1}{s'}}\\
	&+&\dfrac{1}{t} \fint_{B(x,t)} \left| (\tilde{B}_j \overline{W_j})(y) \right| dy,
	\end{eqnarray*}
where in the first inequality we used the definition of the operator $M^{loc}_{W^{-1,s}}$ to some $s$ to be chosen later. 

Consider $1<\alpha<p$ and $1<\beta <p'$, analogous in the proof of Theorem \ref{prop_curl=0const}. Applying the H\"older's inequality and the Sobolev-Poincar\'e inequality where $\beta'=\alpha^{\ast}$ we have
	\begin{eqnarray*}
		\dfrac{1}{t} \fint_{B(x,t)} \left| (\tilde{B}_j \overline{W_j})(y) \right| dy
		&\lesssim& \left( \fint_{B(x,t)} \left| W_j(y)\right|^{\beta} dy \right)^{\frac{1}{\beta}}  \left( \fint_{B(x,t)} \left|\dfrac{1}{t}\big(B_{ij}(y) - (B_{ij})_{B_x^t}\big) \right|^{\beta'} dy \right)^{\frac{1}{\beta'}} \\
		&=& \left( \fint_{B(x,t)} \left| W_j(y)\right|^{\beta} dy \right)^{\frac{1}{\beta}} \left( \fint_{B(x,t)} \left|\dfrac{1}{t}(B_{ij}(y) - (B_{ij})_{B_x^t}) \right|^{\alpha^*} dy \right)^{\frac{1}{\alpha^*}} \\
		&\lesssim& \left(  \fint_{B(x,t)} \left| W_j(y)\right|^{\beta} dy \right)^{\frac{1}{\beta}} {\left( \fint_{B(x,t)} \left|\nabla B_{ij}(y) \right|^{\alpha} dy \right)^{\frac{1}{\alpha}}}.
	\end{eqnarray*}
Plugging this inequality at previous control and taking the supremum for $0<t<1$ we have
\begin{align*}
m_{\varphi}(V \cdot W)(x)&\lesssim \sum_{i<j} M^{loc}_{W^{-1,s}}(({\rm curl}_{\L} \; W)_{ij})(x) \left( M \left(| \nabla \tilde{B}_{ij}|^{s'}\right)(x) \right)^{\frac{1}{s'}} \\
&+\sum_{i,j=1}^{n} \left( M(\left| W_j \right|^{\beta} )(x) \right)^{\frac{1}{\beta}}  \left( M(\left|\nabla B_{ij} \right|^{\alpha})(x) \right)^{\frac{1}{\alpha}}.
\end{align*}

Taking the same choice of $s$ in the previous theorem (in fact, just replace $p$ by $p'$ in the mentioned calculations) and using the H\"older's inequality, we may conclude

	\begin{eqnarray*}
		\|V \cdot W \|_{h^{1}}
		&\lesssim& \sum_{i,j=1}^{n} \left( \norma{M^{loc}_{W^{-1,s}}(({\rm curl}_{\L} \; W)_{ij})}_{ L^{p'}} + \norma{ W_j}_{L^{p'}} \right) \norma{ \nabla B_{ij}}_{L^{p}} \\
		&\lesssim& \sum_{i,j=1}^{n} \left(\|{\rm curl}_{\L} \; W)_{ij}\|_{ L^{p'}} + \norma{ W_j}_{L^{p'}} \right)\norma{ \nabla B_{ij}}_{L^{p}}.
	\end{eqnarray*}
From the definition of $B_{ij}$ we have $\norma{\nabla B_{ij}}_{L^p} \lesssim \norma{U}_{W^{2,p}} \lesssim \norma{V}_{L^p}$, thus
		
	$$\norma{V \cdot W}_{h^1} \lesssim \left( \norma{W}_{L^{p'}} + \sum_{i,j} \norma{({\rm curl}_{\L} \; W)_{ij}}_{L^{p'}} \right)\norma{V}_{L^p},$$
as desired. 
\qed

\section{Proof of the Theorem B}\label{S6} 
Let $g \in bmo(\R^N)$ 
and assume $f \in (\D C_{\L})^p_{1,0} \cup (\D C_{\L})^p_{0,1}  \subset h^{1}(\R^{N})$ from Theorem A. By the duality $bmo(\R^N)=(h^1(\R^N))^{\ast}$ follows 
$$\left|\int_{\R^N} g(x)\overline{f(x)}dx \right|\leq C \norma{g}_{bmo}, \quad \forall \,  f \in (\D C_{\L})^p_{1,0} \cup (\D C_{\L})^p_{0,1}. $$
So now, it is sufficient to prove that 
$$\norma{g}_{bmo} \leq C \sup_{f \in X} \left| \int_{\R^N} g(x)\overline{f(x)}dx \right|$$ 
for $X= (\D C_{\L})^p_{1,0}$ or $X=(\D C_{\L})^p_{0,1}$.
In order to estimate  $\norma{g}_{bmo}$, from the definition in \eqref{bmo}, we split in two cases : balls $B:=B(x_0, R)$ with $R\leq 1$ and $R >1$. 

Let $B^{*}:=B(x_0,2R)$. The Theorem III.2 in \cite{CLMS} asserts that 
	\begin{equation*}
		\left( \fint_{B} |g(x)-g_B|^2 dx \right)^{\frac{1}{2}} \leq C \; \sup_{V, W} \left| \int g(x) (V \cdot W)(x)dx \right|,
	\end{equation*}
	where the supremum is taken over all real vector fields $V, W$ in $C_c^\infty(B^{*})$, with $\norma{V}_{L^2}, \norma{W}_{L^{2}} \leq 1$, satisfying ${\rm div} \; V = 0$ and ${\rm curl} \; W = 0$. We will adapt this argument in our setting. It follows by \cite[Corollary 2.1, pp. 20]{GR} and Lemma \ref{lemma_normaL_W} that
	\begin{equation}\label{volta}
		\norma{g-g_B}_{L^2(B)}  \lesssim \norma{{\nabla g}}_{W^{-1,2}(B)} \lesssim \sum_{i=1}^n \norma{L_i^* g}_{W^{-1,2}(B)} = \sup_{\substack{\norma{\nabla u}_{L^2(B)} \leq 1\\u \in C^\infty_c(B)}} \left| \int g(x) \; \overline{L_i u}(x)dx \right|.
	\end{equation}


We claim that for each $u \in C_c^\infty(B)$ with $\norma{\nabla u}_{L^2(B)} \leq 1$ and $1<p<\infty$ there exist vector fields $V,W$ satisfying $div_{\L^*} \; V = 0$ with $\norma{V}_{L^p}  \leq 1$ and $curl_{\L} \; W = 0$ with $\norma{W}_{L^{p'}}  \leq 1$ such that 
\begin{equation}\label{identVW}
V \cdot W = C |B|^{-\frac{1}{2}} \overline{L_iu}, 
\end{equation} 
for some constant $C>0$. Plugging into \eqref{volta} {we have}
	\begin{equation}\label{parte1}
		\left( \fint_{B} |g(x)-g_B|^2 dx \right)^{\frac{1}{2}} \lesssim {\sum_{i=1}^n \norma{L_i^* g}_{W^{-1,2}(B)}  \lesssim \; \sup_{f \in (\D C_{\L})^p_{1,0} \cap (\D C_{\L})^p_{0,1}} \left| \int g(x)\overline{f(x)} dx \right|}.
	\end{equation}

Consider a function $u \in C_c^\infty(B)$ with $\norma{\nabla u}_{L^2(B)} \leq 1$ and $\eta \in C_c^\infty(B(0,2))$ such that $\eta \equiv 1$ in $B(0,1)$ and $\norma{\eta}_{L^\infty(B(0,2))} \leq 1$. Denote $\eta_B(w) := \eta\left(\dfrac{w-w_0}{R}\right)$ and define the vector fields 
		\begin{equation}\label{def_VW_1}
			V:=\dfrac{ |B|^{\frac{1}{2} - \frac{1}{p}}}{2 C} \left( \overline{L_iu} \; e_j - \overline{L_ju} \; e_i \right) 
\,\,\, \text{and}	\,\,\,		W := \gamma |B|^{-\frac{1}{p'}} \nabla_{\L} \left( \left( {x_j} -{x^0_j} \right)\eta_B(x) \right),
		\end{equation}
for $i,j \in \left\{1,...,n\right\}$ with $i \neq j$, where $\{e_1, \dots , e_n\}$ denotes de canonical basis of $\R^{n}$ and $C,\gamma$ 
are appropriate positive constants to be chosen later. 
We claim that $$V \cdot W = {\dfrac{\gamma}{2C}} |B|^{-\frac{1}{2}} \overline{L_iu},$$ where 
$div_{\L^*} \; V = 0$ with $\norma{V}_{L^p}  \leq 1$ and $curl_{\L} \; W = 0$ with $\norma{W}_{L^{p'}}  \leq 1$, for $1<p \leq 2$. 
Clearly
		$$div_{\L^*} \; V = L_j^*V_j + L_i^*V_i= \frac{|B|^{\frac{1}{2} - \frac{1}{p}}}{2C} \left[ L_i^*, L_j^* \right]\overline{u}=0$$ 
and $|V| \leq |B|^{\frac{1}{2} - \frac{1}{p}} |\nabla u| $ choosing $C:= \displaystyle \max_{\substack{1\leq k\leq m\\1\leq j \leq n}} \left\{1, |a_{jk}| \right\}$.
		Since $\supp(V) \subseteq B$ follows by the Holder's inequality that 
		$\norma{V}_{L^p(B)} \leq |B|^{\frac{1}{p} - \frac{1}{2}} \norma{V}_{L^2(B)} \leq |B|^{\frac{1}{p} - \frac{1}{2}} |B|^{\frac{1}{2} - \frac{1}{p}} \norma{\nabla u}_{L^2(B)} \leq 1.$
		It is easy to see that $curl_{\L} \; W =\gamma |B|^{-\frac{1}{p'}} curl_{\L} (\nabla_{\L} \varphi)=\gamma |B|^{-\frac{1}{p'}} \left([L_{i},L_{j}]\varphi\right)_{ij}=0$. Note that $\supp(W) \subseteq \supp(\eta_B) \subseteq B^{*}$ and $L_\ell \left( x_j - x^0_j \right) = \delta_{\ell j}$. 
Furthermore, for each $x \in B^{*}$ we have
		\begin{eqnarray*}
			\sum_{k=1}^{n} \left|L_k \left( \left({x_j} -{x^0_j} \right)\eta_B(x) \right) \right| &=& \sum_{k=1}^{n} \left| \delta_{kj} \eta_B(x) + \left( {x_j} - {x^0_j} \right) L_k  \eta_B(x)  \right|\\
			&\leq& nC |\eta_B(x)| +  2R \sum_{k=1}^{n} \frac{1}{R} \left|L_k \eta\left(\frac{x-x_0}{R} \right) \right| \\
			&=& nC + 2 \sum_{k=1}^{n} \left|L_k \eta\left(\frac{x-x_0}{R} \right) \right|
		\end{eqnarray*}
and choosing $\gamma := {2^{-\frac{N}{p'}}}( 2\norma{\nabla_{\L}\eta}_{L^\infty} + nC)^{-1}$, follows $|W|\leq 2^{-\frac{N}{p'}} |B|^{-\frac{1}{p'}}$ that implies
		\begin{eqnarray*}
			\norma{W}_{L^{p'}} &\leq& |B^{\ast}|^{\frac{1}{p'}} \norma{W}_{L^{\infty}(B^{*})}
			= 2^{\frac{N}{p'}} |B|^{\frac{1}{p'}} \norma{W}_{L^{\infty}(B^{\ast})}
			\leq 2^{\frac{N}{p'}} |B|^{\frac{1}{p'}} \; 2^{-\frac{N}{p'}} |B|^{-\frac{1}{p'}} =1.
		\end{eqnarray*}
		Lastly, we point out that $V \cdot W = V_i \overline{W_i} + V_j \overline{W_j}$ and $V_i=V_j=0$ on $\R^{N}\backslash B$. Furthermore, as $\eta_B \equiv 1$ in $B$ then for each $x \in B$ we have
		$$W_k(x) = \gamma |B|^{-\frac{1}{p'}} \left( \delta_{kj} \eta_B(x) + (x_j - x_j^0) L_k \eta_B (x)  \right) = \gamma |B|^{-\frac{1}{p'}} \delta_{kj} $$
for $k=i,j$. In particular, as we are assuming $\L$ as in \eqref{Lsimples} thus $W_i=0$ and $W_j = \gamma |B|^{-\frac{1}{p'}}$ on $B$. Therefore,
		\begin{eqnarray*}
			V \cdot W &=& V_j\; W_j = {\dfrac{|B|^{\frac{1}{2} - \frac{1}{p}}}{2C}} \overline{L_iu} \; \gamma |B|^{-\frac{1}{p'}} = {\dfrac{\gamma}{2C}} |B|^{-\frac{1}{2}}\overline{L_iu}.
		\end{eqnarray*}
Now we adapt the previous construction to attend $p > 2$. Consider the vector fields  
		\begin{equation}\label{def_VW_2}
			V = \gamma' |B|^{-\frac{1}{p}} \left[L_i^* \left(\eta_B(x) \left(x_j - x_j^0 \right) \right) e_j - L_j^* \left(\eta_B(x) \left(x_j - x_j^0 \right) \right) e_i \right] \text{ and } W = \frac{|B|^{\frac{1}{2} - \frac{1}{p'}}}{C} \nabla_{\L} u,
		\end{equation}
with $\gamma',C$ are  appropriate constants to be chosen. Analogously as proved before, we have $div_{\L^*} \; V = curl_{\L} \; W = 0$. Clearly, $\supp(V) \subset B^{*}$ and since
\begin{align*}
L_{\ell}^{*} \left(\eta_B(x) \left( x_j - x_j^0 \right)\right)&=\left( x_j - x_j^0 \right)L_{\ell}^{*}\eta_B(x)+\eta_B(x)L_{\ell}^{*}\left( x_j - x_j^0 \right)\\
&=\frac{\left( x_j - x_j^0 \right)}{R}\left(L_{\ell}^{*}\eta_B\right)\left( \frac{x-x_{0}}{R} \right)- \eta_B(x)\delta_{\ell j}
\end{align*}
that implies $|V|\leq \gamma' |B|^{-\frac{1}{p}} \left( \; 1 + 4 \norma{\nabla_{\L^*}\eta}_{L^\infty} \right)=2^{-\frac{N}{p}}|B|^{-\frac{1}{p}}=|B^{*}|^{-\frac{1}{p}}$ and then $\|V\|_{L^{p}}\leq 1$. Since $\supp(W) \subseteq B$ and $1<p'<2$ follows by the Holder's inequality that 
		$\norma{W}_{L^{p'}} \leq C^{-1}|B|^{\frac{1}{p} - \frac{1}{2}} \norma{W}_{L^2} \leq C^{-1} |B|^{\frac{1}{p} - \frac{1}{2}} |B|^{\frac{1}{2} - \frac{1}{p}} \norma{\nabla_{\L} u}_{L^2} \leq 1$, {where $C>0$ is the constant from the control $\|\nabla_{\L}u\|_{L^{2}}\leq C \|\nabla u\|_{L^{2}}$ given by $\displaystyle{C:= {N\sqrt{n}} \max_{\substack{1\leq k\leq m\\1\leq j \leq n}} \left\{ 1,|a_{jk}|\right\}}$.}  
In the same way, 
$V \cdot W = {{\gamma'}{C^{-1}}} |B|^{-\frac{1}{2}} \overline{L_iu}$. Indeed, $V \cdot W = V_i \overline{W_i} + V_j\overline{W_j}$ and now $W_i=W_j=0$ on $\R^{N}\backslash B$. 
As $\eta_B \equiv 1$ in $B$ then for each $x \in B$ we have
	$$V_k(x) = \gamma' |B|^{-\frac{1}{p}} \left( \delta_{ki}\eta_B (x)\delta_{jj} - \delta_{kj} \eta_B(x)\delta_{ij} \right) = \gamma' |B|^{-\frac{1}{p}} \delta_{ki} $$
	and $W_k = C^{-1} |B|^{\frac{1}{2}-\frac{1}{p'}}L_k u$, for $k=i,j$. Therefore,
	\begin{eqnarray*}
		V \cdot W &=& V_i \overline{W_i} = \gamma' |B|^{-\frac{1}{p}} C^{-1} |B|^{\frac{1}{2}-\frac{1}{p'}}\overline{ L_i u} = \dfrac{\gamma'}{C} |B|^{-\frac{1}{2}} \overline{L_iu}.
	\end{eqnarray*}

	
We conclude the identity \eqref{identVW} taking $C:= \max \{\gamma, \gamma' \}$. {We remark that \eqref{parte1} holds for any ball  $B$}.
	
Now we moving on assuming {$R\geq1$}. We claim that  
	\begin{equation}\label{eq_p}
		\left( \fint_{B(x_0,R)} |g(w)|^{{p}} dw \right)^{\frac{1}{p}} \leq C \sup_{f \in (\D C_{\L})^p_{0,1} } \left|\int g(x)\overline{f(x)}dx\right|.
	\end{equation}
Firstly, we will prove the control \eqref{eq_p} when $B = B(0,1)$, denoted by $B_1$. {It follows by \cite[Theorem 1, pp. 108]{Ne}  that the inequality
	\begin{equation}\label{weak}
		\norma{g}_{L^r(B_1)} \leq C \left[ \norma{g}_{W^{-1,r}(B_1)} + \sum_{i=1}^n \norma{L_i^* g}_{W^{-1,r}(B_1)} \right], 
	\end{equation}
	holds for any $1<r<\infty$. 
The estimates for $\norma{L_i^* g}_{W^{-1,p}(B_1)}$ are analogous to those presented in \eqref{parte1} replacing $W^{-1,2}(B_1)$ by $W^{-1,p}(B_1)$. 
In fact, we claim that for each $u \in C_c^\infty(B_1)$ with $\norma{\nabla u}_{L^{p'}(B_1)} \leq 1$ and $1<p'<\infty$ there exist vector fields $V,W$ satisfying $div_{\L^*} \; V = 0$ with $\norma{V}_{L^p}  \leq 1$ and $curl_{\L} \; W = 0$ with $\norma{W}_{L^{p'}}  \leq 1$ such that 
\begin{equation}\label{identVW_p}
	V \cdot W = \widetilde{C} |B_1|^{-\frac{1}{p}} \overline{L_iu}, 
\end{equation} 
for some constant $\widetilde{C}>0$ and then 
\begin{equation}\label{parte1_p}
	\sum_{i=1}^n \norma{L_i^* g}_{W^{-1,p}(B_1)}  \lesssim \; \sup_{f \in (\D C_{\L})^p_{1,0} \cap (\D C_{\L})^p_{0,1}} \left| \int g(x)\overline{f(x)} dx \right|.
\end{equation}
As before, consider a function $u \in C_c^\infty(B_1)$ with $\norma{\nabla u}_{L^{p'}(B_1)} \leq 1$ and $\eta \in C_c^\infty(B_{1}^*)$ such that $\eta \equiv 1$ in $B_{1}$ and $\norma{\eta}_{L^\infty(B_{1}^*)} \leq 1$. Define the vector fields 
\begin{equation}\label{def_VW_3}
	V = \gamma' |B_1|^{-\frac{1}{p}} \left[L_i^* \left(\eta(x) x_j \right) e_j - L_j^* \left(\eta(x) x_j \right) e_i \right] \text{ and } W = C^{-1} \nabla_{\L} u,
\end{equation}
for $i,j \in \left\{1,...,n\right\}$ with $i \neq j$, 
and $C,\gamma'$
are appropriate positive constants to be chosen later. Analogously as proved before, we have $div_{\L^*} \; V = curl_{\L} \; W = 0$. Clearly, $\supp(V) \subset B_{1}^*$ and since
\begin{align*}
	L_{\ell}^{*} \left(\eta(x) x_j \right)&= x_jL_{\ell}^{*}\eta(x) + \eta(x)L_{\ell}^{*} x_j = x_jL_{\ell}^{*}\eta(x) - \eta(x)\delta_{\ell j}
\end{align*}
we have $|V|\leq \gamma' |B_1|^{-\frac{1}{p}} \left( \; 1 + 4 \norma{\nabla_{\L^*}\eta}_{L^\infty} \right)=|B_{1}^*|^{-\frac{1}{p}}$, choosing $\gamma'= 2^{-\frac{N}{p}}( 1+ 4\norma{\nabla_{\L}\eta}_{L^\infty} )^{-1}$, that implies $\|V\|_{L^{p}}\leq 1$. Taking the constant from the control $\|\nabla_{\L}u\|_{L^{p'}}\leq C \|\nabla u\|_{L^{p'}}$ given by $\displaystyle{C:= {N\sqrt{n}} \max_{\substack{1\leq k\leq m\\1\leq j \leq n}} \left\{ 1,|a_{jk}|\right\}}$, then $\norma{W}_{L^{p'}} \leq C^{-1} \norma{\nabla_{\L} u}_{L^{p'}} \leq 1$. \\
To prove
$V \cdot W = {{\gamma'}{C^{-1}}} |B_1|^{-\frac{1}{p}}\overline{L_iu}$, note that $V \cdot W = V_i \overline{W_i} + V_j\overline{W_j}$ and $W_i=W_j=0$ on $\R^{N}\backslash B_{1}$.
As $\eta \equiv 1$ in $B_{1}$ then for each $x \in B_{1}$ we have
$$V_k(x) = \gamma' |B_{1}|^{-\frac{1}{p}} \left( \delta_{ki}\eta (x)\delta_{jj} - \delta_{kj} \eta(x)\delta_{ij} \right) = \gamma' |B_{1}|^{-\frac{1}{p}} \delta_{ki} $$
and $W_k = C^{-1} L_k u$, for $k=i,j$. Therefore,
\begin{eqnarray*}
V \cdot W &=& V_i \overline{W_i} = \gamma' |B|^{-\frac{1}{p}} C^{-1} \overline{ L_i u} = \dfrac{\gamma'}{C} |B|^{-\frac{1}{p}} \overline{L_iu}.
\end{eqnarray*}
We conclude the identity \eqref{identVW_p} taking $\widetilde{C}:=\dfrac{\gamma'}{C} |B|^{-\frac{1}{p}}$.
	
	\begin{lemma}\label{lema_phi}
		If $\phi \in C^\infty_c(B(0,1))$ then we can write
		$\phi = V_1 \cdot W_1$, 
		where $V_1, W_1$ 
		are smooth vector fields satisfying the following properties: 
		\begin{enumerate}
			\item[(i)] $\supp V_1 \subset B(0,1)$ and $\supp W_1 \subset B(0,2)$;
			\item[(ii)] $curl_{\L} \; W_1 = 0$ and $\norma{W_1}_{L^{p'}} \leq C_1$, for some $C_{1}>0$ independent of $\phi$;
			\item[(iii)] $\norma{V_1}_{L^{p}} = \norma{\phi}_{L^{p}}$ with $\norma{div_{\L^*} \; V_1}_{L^{p}} \leq \norma{\nabla_{\L^*} \phi}_{L^{p}}$.
		\end{enumerate}
Analogously, we may write $\phi = V_2 \cdot W_2$, where $V_2, W_2$ 
		are smooth vector fields satisfying: 	
	\begin{enumerate}
			\item[(iv)] $\supp W_2 \subset B(0,1)$ and $\supp V_2 \subset B(0,2)$;
		         \item[(v)] $div_{\L^*} \; V_2 = 0 $ and $\norma{V_2}_{L^p} \leq C_2$,  for some $C_{2}>0$ independent of $\phi$;
			\item[(vi)] $\norma{W_2}_{L^{p'}} = \norma{\phi}_{L^{p'}}$ with $\norma{curl_{\L} \; W_2}_{L^{p'}} \leq 2\norma{\nabla_{\L} \phi}_{L^{p'}}$;
		\end{enumerate}
	\end{lemma}

A direct consequence of the previous lemma show that  for each $\phi \in C^\infty_c(B(0,1))$ with $\norma{ \phi}_{W^{1,p'}} \leq 1$, there exists a constant $C_{2}>0$ independent of $\phi$ such that $C_{2}\phi=V_{2}\cdot W_{2} \in (\D C_{\L})^p_{0,1}$ for $1<p<\infty$. Then for $B_{1}:=B(0,1)$ we have  
\begin{eqnarray}\label{substituicao_p}
		\norma{g}_{W^{-1,p}(B_{1})} 
		= \sup_{\substack{\norma{\phi}_{W^{1,p'}(B_{1})}  \leq 1\\\phi \in C^\infty_c(B_{1})}} \left| \int g(x) \; \overline{\phi(x)}dx \right|
		&=& (C_{2})^{-1} \sup_{\substack{\norma{\phi}_{W^{1,p'}(B_{1})} \leq 1\\\phi \in C^\infty_c(B_{1})}} \left| \int g(x) \; \overline{(V_2 \cdot W_2)}(x)dx \right|  \nonumber \\
		&\leq& (C_{2})^{-1} \sup_{f \in (\D C_{\mathcal{\L}})^p_{0,1}} \left| \int g(x) \; \overline{f(x)}dx \right|.
	\end{eqnarray}
Using the first part of the lemma, the previous control follows the same replacing $(\D C_{\mathcal{\L}})^p_{0,1}$ by $(\D C_{\mathcal{\L}})^p_{1,0}$, that is,
\begin{equation}\label{substituicao_p'}
	\norma{g}_{W^{-1,p'}(B_{1})} \leq (C_{1})^{-1} \sup_{f \in (\D C_{\mathcal{\L}})^p_{1,0}} \left| \int g(x) \; \overline{f(x)}dx \right|.
\end{equation}

	\begin{proof}
		Fix $\phi \in C^\infty_c(B_{1})$ and $\eta \in C_c^\infty(B_{1}^*)$ such that $\eta \equiv 1$ in $B_{1}$ and $\norma{\eta}_{L^\infty(B_{1}^*)} \leq 1$. We define
		$V_1(x) := \phi(x)e_1$ and $W_1(x) := \nabla_{\L} \left( x_{1} \eta(x) \right)$. Clearly $curl_{\L} \; W_1 = 0$, $\norma{V_1}_{L^{p}} = \norma{\phi}_{L^{p}}$ and $\|div_{\L^*} \; V_1\|_{L^{p}}=\|L_{1}^{*}\phi\|_{L^{p}}\leq  \norma{\nabla_{\L^*} \phi}_{L^{p}}$.    Note that for $x \in B_{1}$ we have 
$$L_{1}[x_{1}\eta(x)]=\eta(x)+x_{1}[L_{1}\eta](x)=1,  $$		
since  $\supp V_1 \subset B_{1}$ we have $\left(V_1 \cdot W_1\right)(x) =\phi(x)L_{1}[x_{1}\eta(x)]=\phi(x)$. Moreover 		
		$$|W_1(x)| = \sum_{j=1}^n |L_j(x_{1}\eta(x) )| \leq |L_jx_{1}||\eta| + |x_{1}| \sum_{j=1}^n |L_j\eta(x)| \lesssim |\eta(x)| + |x_{1}| |\nabla_{\L} \eta(x) |$$
and as $\supp W_1 \subset B_{1}^*$, we have 
		$$\norma{W_1}_{L^{p'}} \leq |B_{1}^*|^{\frac{1}{p'}}\|W_{1}\|_{L^{\infty}} \leq  |B_{1}^*|^{\frac{1}{p'}}   \left(1+2\norma{\nabla_{\L} \eta}_{L^{\infty}} \right)$$ 
For the second part we define $V_2(x) = L_1^* \left( x_1 \eta(x) \right)e_2 - L_2^* \left( x_1 \eta(x) \right)e_1$, $W_2(x)=\phi(x) e_2$  that satisfies (by definition) $\norma{W_2}_{L^{p'}} = \norma{\phi}_{L^{p'}}$, $\norma{curl_{\L} \; W_2}_{L^{p'}} \leq 2\norma{\nabla_{\L} {\phi}}_{L^{p'}}$ and $div_{\L^*} \; V_2 = 0$. Since 
$$L^{\ast}_{\ell}[x_{1}\eta(x)]=\delta_{\ell 1}(x)\eta(x)+x_{1}[L^{\ast}_{\ell}\eta](x)  $$		
we have 
		$ |V_2(x) | \leq |x_1| \left( |L_1^*\eta(x)| + |L_2^*\eta (x)| \right) + |\eta(x)|$
and $\supp V_2 \subset B_{1}^*$ that implies $\norma{V_2}_{L^{p}} \leq \left( 2\norma{\nabla_{\L^*} \eta}_{L^{\infty}} +1 \right)|B_{1}^*|^{\frac{1}{p}}$. 
Note that $\supp W_2 \subset B_{1}$ and since $L^{\ast}_{\ell}[x_{1}\eta(x)]=1$ for $x \in B_{1}$ we have  $\phi=V_2 \cdot W_2$. 
		\begin{flushright}
			\qed
		\end{flushright}
	\end{proof}

	Now, we moving on for a ball $B(x_0, R)$, with $R \geq 1$. For each $\phi \in C^\infty_c(B(x_{0},R))$ we may define  $\widetilde{\phi} \in C^\infty_c(B(0,1))$ given by $\widetilde{\phi}(y):={\phi}\left(x_{0}+yR\right)$ and applying the Lemma \ref{lema_phi} there exists vector fields $\widetilde{V_{i}}, \widetilde{W_{i}}$ for $i=1,2$ satisfying (i)-(vi) above such that $\widetilde{\phi}=\widetilde{V_{i}}\cdot \widetilde{W_{i}}$. Defining ${V_i}(x) := R^{-\frac{N}{p}}\widetilde{V_i}\left( \frac{x-x_{0}}{R} \right)$ and ${W_i}(x) := R^{-\frac{N}{p'}}\widetilde{W_i}\left( \frac{x-x_{0}}{R} \right)$  we have that there exist constants $C_{i}>0$ independent of ${\phi}$ such that $C_{1}{\phi}={V_{1}} \cdot {W_{1}} \in (\D C_{\mathcal{L}})^{p}_{1,0}$ and  $C_{2}{\phi}={V_{2}} \cdot {W_{2}} \in (\D C_{\mathcal{L}})^{p}_{0,1}$.

For each ${g} \in L^{1}_{loc}(\R^{N})$, we define $\widetilde{g}(y) = {g}(x_0+Ry)$ and then 
	\begin{eqnarray}\label{ident}
		\int_{B(x_0,R)} {g}(x) \overline{(V_i \cdot{W_i}) (x)} \; dx 
		&=& \int_{B(0,1)} \widetilde{g}(y) \overline{(\widetilde{V_i} \cdot \widetilde{W_i})(y)} \; dy.
	\end{eqnarray}
	Furthermore, using change of variables and the inequality \eqref{weak} for $B_{1}:=B(0,1)$ we have
	\begin{align*}
	\left( \fint_{B(x_0,R)} |{g}(x)|^p dx \right)^{\frac{1}{p}} =\left( \fint_{B_{1}} |\widetilde{g}(y)|^p dy \right)^{\frac{1}{p}} &= C_N \norma{\widetilde{g}}_{L^p(B_{1})} \\
		&\leq C \left[ \norma{\widetilde{g}}_{W^{-1,p}(B_{1})} + \sum_{i=1}^n \norma{L_i^* \widetilde{g}}_{W^{-1,p}(B_{1})} \right]
	\end{align*}
From \eqref{substituicao_p} and the identity \eqref{ident} we have
\begin{eqnarray}
		\norma{\widetilde{g}}_{W^{-1,p}(B_{1})} 
		\lesssim \sup_{f \in (\D C_{\mathcal{\L}})^p_{0,1}} \left| \int g(x) \; \overline{f(x)}dx \right|
	\end{eqnarray}
and 
by the inequality \eqref{parte1_p} we have
	\begin{equation}\label{parte2}
 \sum_{i=1}^n \norma{L_i^* \widetilde{g}}_{W^{-1,p}(B_{1})}  \lesssim \; \sup_{f \in (\D C_{\L})^p_{1,0} \cap (\D C_{\L})^p_{0,1}} \left| \int g(x)\overline{f(x)} dx \right|.
	\end{equation}

Combining the previous estimate, we may conclude 
\begin{eqnarray*}
	\norma{g}_{bmo} 
	&\leq& \sup_{|B(x_{0},R)|\leq1} \left( \fint_{B(x_0,R)} |g(x)-g_B|^2dx \right)^{\frac{1}{2}} + \displaystyle \sup_{|B(x_{0},R)|>1} \left( \fint_{B(x_0,R)} |g(x)|^pdx \right)^{\frac{1}{p}} \\
	&\lesssim& \left( \sup_{f \in (\D C_{\L})^p_{0,1} \cap (\D C_{\L})^p_{1,0}} \left| \int g(x)\overline{f(x)}dx\right| + \sup_{f \in (\D C_{\L})^p_{0,1} } \left| \int g(x)\overline{f(x)}dx  \right| \right) \\
	&\lesssim& \sup_{f \in (\D C_{\L})^p_{0,1}} \; \left| \int_{\R^N} g(x)\overline{f(x)}dx \right|.
\end{eqnarray*}  	

The same arguments holds replacing $ (\D C_{\L})^p_{0,1}$ by  $(\D C_{\L})^p_{1,0}$ taking $p'$ instead $p$ in \eqref{weak}. 
\qed

\subsection{Proof of Corolary \ref{coro1}}

To simplify the notation, consider $V:=(\D C_{\L})^p_{1,0}$ and $F:=h^{1}(\R^{N})$. A direct consequence of Theorem A implies that $V$  is a bounded symmetric (i.e. $h \in V$ then $-h \in V$) subset of F. If we prove that the closure of $V$ in the norm $F$, denoted by $\overline{V}$, contains the unit ball of F, follows from Lemma III.1 in \cite{CLMS}, let 
each $\|f\|_{h^{1}}\leq 1$ can be decomposed by
\begin{equation}
f=\sum_{k=1}^{\infty}2^{-k}f_{k}, \quad f_{k} \in V, 
\end{equation}   
with convergence in F. Now, from Lemma III.2 in \cite{CLMS}, the closed convex hull $\widetilde{V}$ contains the unit ball of F if and only if $\|g\|_{(h^{1})^{\ast}}$ is equivalent to the functional 
$$\sup_{f \in V}\left| \int_{\R^{N}}g(x)f(x)dx \right|,$$
that is exactly the conclusion of Theorem B, since $(h^{1}(\R^{N}))^{\ast}=bmo(\R^{N})$. The decomposition \eqref{coroC} follows taking $\lambda_{k}:=2^{-k}\|f\|_{h^{1}} \in \ell^{1}(\C)$, for every $f \in h^{1}(\R^{N})$. Clearly $\|\lambda\|_{\ell^{1}}\leq \|f\|_{h^{1}}$ and the convergence in \eqref{coroC} holds also in the sense of tempered distributions. The same conclusion holds for $V=(\D C_{\L})^p_{0,1}$.

\end{document}